\documentclass[12pt]{amsart}

\usepackage{amsmath, amssymb, amsthm, amsfonts, marginnote, stmaryrd, tikz-cd}

\usepackage{color}

\input xy
\xyoption{all}

\usepackage{hyperref}

\theoremstyle{plain}
\newtheorem{theorem}[subsubsection]{Theorem}

\newtheorem{prop}[subsubsection]{Proposition}

\theoremstyle{definition}

\newtheorem{remark}[subsubsection]{Remark}

\newtheorem{ex}[subsubsection]{Example}

\swapnumbers
\numberwithin{equation}{section}

\def\sfG{\mathsf{G}}
\def\sfG{\mathsf{G}}

\def\sfL{\mathsf{L}}

\def\AA{\mathbb{A}}

\def\CC{\mathbb{C}}

\def\GG{\mathbb{G}}
\def\BG{\mathbb{G}}

\def\RR{\mathbb{R}}
\def\BR{\mathbb{R}}

\newcommand\cF{\mathcal{F}}

\newcommand\cN{\mathcal{N}}
\newcommand\cO{\mathcal{O}}

\newcommand\frb{\mathfrak{b}}
\newcommand\frc{\mathfrak{c}}

\newcommand\frg{\mathfrak{g}}

\newcommand\frakm{\mathfrak{m}}
\newcommand\frn{\mathfrak{n}}

\newcommand\frt{\mathfrak{t}}

\newcommand{\Bun}{\textup{Bun}}

\newcommand{\Ind}{\textup{Ind}}
\newcommand{\ind}{\textup{ind}}

\newcommand\Loc{\textup{Loc}}

\newcommand{\reg}{\textup{reg}}

\newcommand\res{\textup{res}}

\newcommand\Spec{\textup{Spec }}

\newcommand\Sym{\textup{Sym}}

\newcommand{\univ}{\textup{univ}}

\newcommand{\Vect}{\textup{Vect}}

\newcommand\Aut{\textup{Aut}}
\newcommand\Hom{\textup{Hom}}

\newcommand\uHom{\underline{\Hom}}

\newcommand{\Ad}{\textup{Ad}}

\renewcommand\a\alpha
\renewcommand\b\beta
\newcommand\g\gamma
\renewcommand\d\delta
\newcommand\D\Delta

\renewcommand{\l}{\lambda}
\renewcommand{\L}{\Lambda}
\newcommand{\om}{\omega}

\newcommand\quash[1]{}

\newcommand\pt{\textup{pt}}

\newcommand{\Coh}{\textup{Coh}}

\newcommand{\Sh}{\mathit{Sh}}

\newcommand{\beq}{\begin{equation}}
\newcommand{\eeq}{\end{equation}}

\newcommand{\ssupp}{\mathit{ss}}

\newcommand\ver{\mathit{vert}}

\newcommand{\Real}{\operatorname{Re}}

\title{The Whittaker functional is a shifted microstalk}
\author{David Nadler}

\author{Jeremy Taylor}

\address{Department of Mathematics\\University of California, Berkeley\\Berkeley, CA  94720-3840}
\email{nadler@math.berkeley.edu}
\email{jeretaylor@berkeley.edu}

\subjclass[]{}

\dedicatory{}

\begin{document}


\maketitle

\begin{abstract}
For a smooth projective curve $X$ and reductive group $G$, the Whittaker functional on nilpotent sheaves on $\Bun_G(X)$ is expected to correspond to global sections of  coherent sheaves on the spectral side of Betti geometric Langlands. We prove that the Whittaker functional calculates the shifted microstalk of nilpotent sheaves at the point in the Hitchin moduli where the Kostant section intersects the global nilpotent cone.
In particular, the shifted Whittaker functional is exact for the perverse $t$-structure and commutes with Verdier duality.
Our proof is topological and depends  on the intrinsic local hyperbolic symmetry of $\Bun_G(X)$. It is an application of a general result relating vanishing cycles  to the composition of  
restriction to an attracting locus followed by vanishing cycles.

\end{abstract}

\tableofcontents

\section{Introduction}

Let    $X$ be a smooth projective complex curve, $G$ a complex reductive group with Langlands dual $G^\vee$.
 
\subsection{Main result}
The Betti variant~\cite{BN}  of the geometric Langlands conjecture~\cite{Lgl, BD, AG} says there should be an equivalence
\beq
\Sh_{\Lambda}(\Bun_G(X)) \simeq \Ind\Coh_{\mathcal{N}}(\Loc_{G^{\vee}}(X))
\eeq
compatible with  natural structures (Hecke operators, parabolic induction, cutting and gluing curves...) on each side.

In this paper, we will only be concerned with the left hand automorphic side, $\Sh_{\Lambda}(\Bun_G(X))$, the (dg derived) category of sheaves with singular support~\cite{KS} in $\Lambda$. Here $\Lambda = h^{-1}(0) \subset T^*\Bun_G(X)$
is the global nilpotent cone, the closed conic Lagrangian ~\cite{L, G, BD} given by the zero-fiber of the Hitchin system $h:T^*\Bun_G(X)\to \frc^*_G(X)$.

The Kostant section $\kappa: \frc^*_G(X) \to T^*\Bun_G(X)$ to the Hitchin system has image $K = \kappa( \frc^*_G(X) )$ a closed (non-conic) Lagrangian that intersects  $\L$ transversely at a smooth point $\l \in \L$.  Informally speaking,  following paradigms from  $T$-duality applied to Hitchin systems, one expects the ``Lagrangian A-brane" $K$ to correspond to the space-filling ``coherent B-brane" $\cO_{\Loc_{G^{\vee}}(X)}$. If $K$ were to define a nilpotent sheaf $\cF_K\in \Sh_{\Lambda}(\Bun_G(X))$, one would expect the corepresented functor $\Hom(\cF_K, -)$ to give the  microstalk of nilpotent sheaves at the intersection point  $\l = K \cap \L$.
Thus under $T$-duality, one would expect  the microstalk at $\l$  to correspond to the global sections functor $\Gamma(\Loc_{G^{\vee}}(X), -) \simeq \Hom(\cO_{\Loc_{G^{\vee}}(X)}, -)$.

Explicitly describing the nilpotent sheaf $\cF_K$ corepresenting the microstalk of nilpotent sheaves at $\l$ is a difficult problem. But our main result, Theorem~\ref{whitmicrothm}, confirms  the traditional  Whittaker functional
\beq
\xymatrix{
\phi_{f, \rho^{\vee}(\omega)} i^!: \Sh_{\Lambda}(\Bun_G(X)) \ar[r] &  \Vect
}
\eeq
indeed calculates the  (shifted)  microstalk  at $\l = df(\rho^{\vee}(\omega))$ of nilpotent sheaves.
 Here we first pull back along \beq i: \Bun_N^{\omega}(X) \to \Bun_G(X)\eeq and then take vanishing cycles for a particular function \beq f: \Bun_N^{\omega}(X) \to \AA^1.\eeq
We will recall the notation and further details in \ref{whit}.\footnote{Let us at least remark here that the expression $df(\rho^{\vee}(\omega))$ makes sense even though $f$ is only a function on $\Bun_N^{\omega}(X)$ because of the natural splitting \beq T^*_{\rho^{\vee}(\omega)} \Bun_G(X) \simeq T^*_{\rho^{\vee}(\omega)} \Bun_{B^-}(X) \oplus T^*_{\rho^{\vee}(\omega)} \Bun_N^{\omega}(X).\eeq}

We also calculate the shift: the Whittaker functional is the usual exact microstalk (with respect to the perverse $t$-structure) after a shift by $\dim_{\rho^{\vee}(\omega)} \Bun_{B^-}(X)$.  
This exactness of the shifted Whittaker functional was recently obtained by
F\ae rgeman-Raskin~\cite{FR} and we were in part motivated by giving a geometric explanation of their results. Our proof of exactness only uses the hyperbolic symmetry of $\Bun_G(X)$, as opposed to tools from geometric Langlands. Our general result, Theorem \ref{main}, may be applicable outside of representation theory. As will be explained in Section \ref{Tame}, our arguments easily extend to the case of tame ramification.

\begin{remark}
The Whittaker functional is corepresented by the (non nilpotent) Whittaker sheaf $i_!f^* \Psi$ where $\Psi$ is a $\BG_m$-equivariant version of the Artin-Schreier sheaf (see for example~\cite{NY}). Let us pretend that the singular support of $f^*\Psi$ were the graph of the differential $\Gamma_{df} \subset T^*\Bun_N^{\omega}(X)$. This is of course nonsense because singular support is a closed \textit{conic} Lagrangian. But it is motivated by the observation that $f^* \Psi$ corepresents vanishing cycles for $f$. Accepting this, we would then expect the singular support of the Whittaker sheaf to be the shifted conormal bundle $T^*_{\Bun_N^{\omega}(X)} \Bun_G(X) + df \subset T^*\Bun_G(X)$ which coincides with the Kostant section $K \subset T^*\Bun_G(X)$.
\end{remark}

\subsection{Overview} Here is a brief overview of the sections of the paper.

 In  Sect.~\ref{commuting}, we establish the general result, Theorem \ref{main}, that in the presence of hyperbolic $\CC^\times$-symmetry on a complex manifold $Y$, the $!$-restriction to the attracting locus $i:Y^{>0} \to Y$ of a point $y_0\in Y$, followed by vanishing cycles $\phi_{f, y_0}$ for a function $f$ on $Y^{>0}$ is naturally isomorphic to vanishing cycles $\phi_{F, y_0}$ for a suitable extension $F$ of $f$ to $Y$. Adjunction provides a natural map
 \beq\label{eq:nat map}
\xymatrix{
  \phi_{f, y_0}  i^!\ar[r] &  \phi_{F, y_0}.
  }\eeq
  To show~\eqref{eq:nat map} is an isomorphism, we  corepresent the respective functionals by  $!$-extensions of constant sheaves on regions $V \subset W \subset Y$.  The cone of the map~\eqref{eq:nat map} is corepresented by the $!$-extension of the ``difference" $k_{W \setminus V}$. We show the cone vanishes on  sheaves with hyperbolic symmetry since $W \setminus V$ is foliated by its intersection with $\RR^{>0}$-orbits entering through its closed $*$-boundary and exiting through its open $!$-boundary.

Next we fix a singular support $\L \subset T^*Y$, and study the vanishing cycles $ \phi_{F, y_0}$ for the extended  function $F$.
In Theorem \ref{main} we defined $F$ to be maximally negative definite in the repelling directions. Now we also ask for the graph of its differential to intersect $\L$ cleanly along smooth points.
In Sect.~\ref{tip}, we show that $\phi_{F, y_0}$ is calcultes microstalk and is exact after a Maslov index shift which we calculate.

In  Sect.~\ref{whit}, we specialize to the situation of $\Bun_G(X)$ and define the Whittaker functional. We add level structure to uniformize (a quasicompact open substack of) $\Bun_G(X)$ by a scheme.
In  Sect.~\ref{hyp}, we recall the intrinsic hyperbolic action from \cite{DG} so as to apply Theorem \ref{main}.
In  Sect.~\ref{kost}, we interpret the shifted conormal bundle as the Kostant slice to see that it intersects the nilpotent cone cleanly. Then we calculate the Maslov index shift in terms of the dimension of the clean intersection.

\subsection{Acknowledgements}
We thank David Ben-Zvi, Joakim F\ae rgeman, Sam Raskin, and Zhiwei Yun for  helpful discussions.

This is an update to the published version with some errors corrected. We are grateful to Swapnil Garg for his careful reading and pointing out of  errors.

We were partially supported by NSF grants DMS-2101466 (DN) and DMS-1646385 (JT).

\section{General results}

\subsection{Some microlocal sheaf theory} Fix a field $k$. By a sheaf of $k$-modules, we will mean an object of the dg derived category of sheaves of $k$-modules.

Let $Y$ be a real analytic manifold and $F:Y\to \RR$ a real-valued smooth function.
We define the vanishing cycles
\beq  
\xymatrix{
\phi_F= (\Gamma_{F \geq 0}(-)) _{F=0} : \Sh(Y) \ar[r] &  \Sh(F = 0)
}
\eeq
 by first $!$-restricting to $F \geq 0$ and then $*$-restricting to $F = 0$. Further $*$-restricting to a point $y_0$ gives a functional
 \beq\label{eq:stalk of van}
 \xymatrix{
  \phi_{F, y_0}(-) = (\Gamma_{F \geq 0}(-))_{y_0}: \Sh(Y) \ar[r] &  \Vect.
  }\eeq
 We view the functional~\eqref{eq:stalk of van} as a measurement of a sheaf associated to the covector $dF(y_0) \in T^*Y$. The singular support $\ssupp(\cF)\subset T^*Y$ of a sheaf $\cF$ on $Y$ is the closure of those covectors $\xi_0 \in T^*Y$  for which there exists a function $F$ with $F(y_0) = 0$, $dF(y_0) = \xi_0$, and  $\phi_{F, y_0}(\cF) \not \simeq 0$.

Let $\Lambda \subset T^*Y$ be a subanalytic closed conic Lagrangian and $\Sh_{\Lambda}(Y)$ the category of sheaves with singular support in $\Lambda$. Recall there is a subanalytic stratification of $Y$ so that any sheaf with singular support in $\L$ is weakly constructible for the stratification. By adjunction, $!$-restriction to $F \geq 0$ is corepresented by the $!$-extension $k_{F \geq 0}$. The stalk at a point $y_0$ is corepresented by the $!$-extension $k_{B'}$ from a sufficiently small open ball $B' \subset Y$ around $y_0$, see Lemma 8.4.7 of \cite{KS}. Therefore vanishing cycles is corepresented by \beq  \xymatrix{\phi_{F, y_0} (-) \simeq \Hom(k_{B' \cap \{F \geq 0\}}, -): \Sh_{\Lambda}(Y) \ar[r] & \Vect }\eeq
where $k_{B' \cap \{F \geq 0\}}$ is $!$-extended along the boundary of the open ball $B'$ and $*$-extended along the closed boundary  $F = 0$.

Proposition 7.5.3 of \cite{KS} says that for a smooth point  $\xi_0 \in \L$, and any function $F$ such that the graph $\Gamma_{dF}$ of its differential intersects $\L$  transversely at $\xi_0$,  the shifted vanishing cycles $\phi_{F, y_0}[\ind/2]$ is independent of $F$ (up non-canonical isomorphism), only depending on $\xi_0$. Here $\ind/2$ denotes half the Maslov index of three Lagrangians in the symplectic vector space $T_{\xi_0}T^*Y$: the tangent to the graph $\Gamma_{dF}$, the tangent to the singular support $\L$, and the tangent to the  cotangent fiber $T_{y_0}^*Y$.
We call $\phi_{F, y_0}[\ind/2]$ the microstalk functional at $\xi_0$.

If $Y$ is  complex analytic, and $f:Y\to \CC$ is holomorphic, then there is a traditional vanishing cycles functor
\beq
 \xymatrix{
 \phi_f: \Sh(Y) \ar[r] &  \Sh(f = 0)}
 \eeq which we normalize so that it is exact with respect to the perverse $t$-structure (see \cite{KS}, Corollary 10.3.13).
Taking the stalk at a point $y_0$ gives a functional
\beq
\xymatrix{
\phi_{f, y_0}: \Sh(Y)  \ar[r] &  \Vect.}
\eeq
If $\L \subset T^*Y$ is complex subanalytic then on $\Sh_{\Lambda}(Y)$ the complex and real vanishing cycles functionals are related by $\phi_{f, y_0} \simeq \phi_{\Real f, y_0}$ (see \cite{KS} Exercise VIII.13). If also the intersection $\Gamma_{df} \cap \Lambda$ is zero-dimensional, then $\phi_f: \Sh_{\Lambda}(Y) \rightarrow \Sh(f = 0)$ takes values in sheaves with zero dimensional support, so after taking the stalk, $\phi_{f, y_0}$ is still exact.

Singular support behaves well under smooth pullback and pushforward along closed embeddings. For a map $\pi: Z\to Y$, consider  
 the natural Lagrangian correspondence
 \beq
 \xymatrix{
  T^*Z & \ar[l]_-{d\pi}  T^*Y \times_Y Z \ar[r]^-{\pi} &  T^*Y.}
  \eeq
If $\pi$ is smooth, then $\ssupp(\pi^!-)  = d\pi(\pi^{-1} (\ssupp(-)))$; if $\pi$ is a closed embedding, then $\ssupp (\pi_!-) =   \pi(d\pi^{-1}(\ssupp(-))) $
 (see \cite[Propositions 5.4.4 and 5.4.5]{KS}).

\subsection{Restriction to the attracting locus then vanishing cycles} \label{commuting}
Let $Y$ be a complex analytic manifold and $\Lambda \subset T^*Y$ be a subanalytic closed conic Lagrangian singular support condition.

It is not true in general pullback along a closed embedding followed by vanishing cycles can be interpreted as vanishing cycles.
\begin{ex}
Let $k_{y = x^2}$ be the pushforward to $\AA^2$ of the constant sheaf on a parabola and let $i: \AA^1 \to \AA^2$ be the inclusion of the axis $y = 0$. Then $i^! k_{y = x^2}$ is a skyscraper at $0 \in \AA^1$ and hence has nonzero vanishing cycles $\phi_{x, 0} i^! k_{y = x^2} \not \simeq 0$. On the other hand,   we have the vanishing $\phi_{x, 0} k_{y = x^2} \simeq 0$
since the level-sets of $x$ are transverse to $y=x^2$.
\end{ex}

Suppose we have a $\CC^{\times}$-action on $Y$. Let $Y^0\subset Y$ denote the fixed locus,  $Y^{\geq 0} \to Y$ the  attracting locus, and  $Y^{\leq 0} \to Y$ the repelling locus.
Suppose $y_0 \in Y^0$ is a fixed point. Let $Y^{> 0} \to Y$ be the attracting locus of $y_0$,   and $Y^{< 0} \to Y$ the repelling locus of $y_0$.

Let $f: Y^{>0} \to \CC$ be a $\CC^{\times}$-equivariant function where $\CC^{\times}$ acts linearly on the target $\CC$ with some weight.
The Whittaker functional
\beq
\xymatrix{
\phi_{f, y_0} i^! : \Sh_{\Lambda}(Y) \ar[r] &  \Vect
}
\eeq
is defined by pulling back along
\beq
\xymatrix{
i: Y^{> 0} \ar[r] &  Y
}
\eeq and then taking vanishing cycles for the function $f$ at the point $y_0$.

We wish to compare the Whittaker functional with directly taking vanishing cycles on $Y$ without pulling back to $Y^{>0}$ first. To define vanishing cycles on $Y$, we need a function on $Y$ extending $f$.

Choose complex coordinate functions $(y^{< 0}_i, y^0_j, y^{> 0}_k)$ so that $\CC^{\times}$ acts by $y^{<0}_i(z \cdot y) = z^{m_i} y^{< 0}_i(y)$ with negative weights $m_i < 0$, by $y^{> 0}_k(z \cdot y) = z^{n_k} y^{> 0}_k(y)$ with positive weights $n_k > 0$, and fixes the coordinates $y^0_j(z\cdot y) = y^0_j(y)$. Consider  distance functions (for some metric) in these coordinates
\beq
d_{<0}(y) = \sum |y_i^{< 0}|^2, \qquad d_{>0}(y) = \sum |y_i^{>0}|^2, \qquad d_0(y) = \sum |y^0_i|^2.
\eeq
Note that acting by $z \in \CC^{\times}$ with $|z| >1$ decreases $d_{<0}$ (when non-zero), increases $d_{>0}$ (when non-zero), and fixes $d_0$.

Now define  (the germ near $y_0$ of)  the real-valued smooth function $F:Y\to \RR$ by the formula \beq\label{FDef} F = \Real f - (d_0 + d_{<0}),\eeq or more precisely, $F(y) = \Real f(\pi(y)) - d_0(y) - d_{<0}(y)$, where $\pi: Y \to Y^{>0}$ is the $\CC^{\times}$-equivariant coordinate  projection.


Note that $dF(y_0)$ lies in the second summand
of   the $\CC^{\times}$-equivariant splitting $T^*_{y_0}Y \simeq T^*_{y_0}Y^{\leq 0} \oplus T^*_{y_0}Y^{> 0}$ and  $dF(y_0) =  d\Real f(y_0) \in T^*_{y_0}Y^{> 0}$.


Having fixed the function $F$, note the vanishing cycles functor
\beq
\xymatrix{
\phi_{F, y_0}(-) \simeq \Hom(k_V, -): \Sh_{\Lambda}(Y) \ar[r] &  \Vect
}
\eeq
 is corepresented by  the $!$-extension $k_V$ of the constant sheaf on
\beq
V := B \cap \{F \geq 0\},
\eeq
where $B\subset Y$ is a small open ball around $y_0$, see Figure \ref{Fig1}. More generally, we can
take $B$ to be a polyball cut out by
\beq
 d_{>0} < a,\qquad d_0 +  d_{< 0}  < b  
\eeq
for $a, b>0$.

\begin{figure}
\includegraphics[scale=0.7, trim=100 450 100 50, clip]{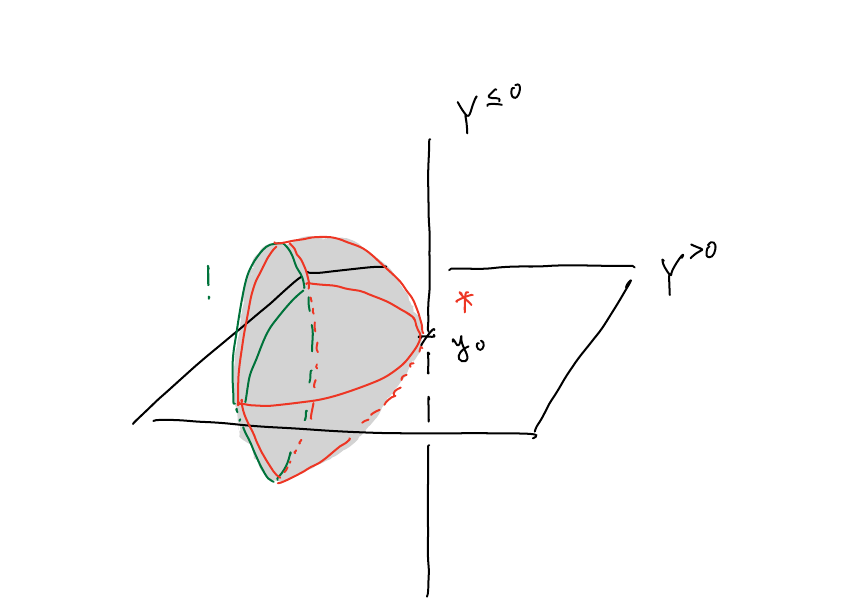}
\caption{Structure of $k_V$ where $V = B \cap \{F \geq 0\}$.}\label{Fig1}
\end{figure}

By adjunction, the $!$-restriction along  $i: Y^{> 0} \to Y$
 is internally corepresented
 \beq
 i^! (-) \simeq \uHom(k_{Y^{>0}}, -)
 \eeq
by  the $!$-extension $k_{Y^{>0}}$ of the constant sheaf on $Y^{>0}$.
So the Whittaker functional is corepresented
\beq \label{WhittakerCorepresented}
\phi_{f, y_0} i^!( -) = \Hom(k_{Y^{>0} \cap B \cap \{\Real f \geq 0\}}, -)
\eeq
by the $!$-extension $k_{Y^{>0} \cap B \cap \{\Real f \geq 0\}}$  
of the constant sheaf on
\beq
U :=  Y^{>0} \cap B \cap \{\Real f \geq 0\},
\eeq
see Figure \ref{Fig2}.

\begin{figure}
\includegraphics[scale=0.7, trim=100 450 100 50, clip]{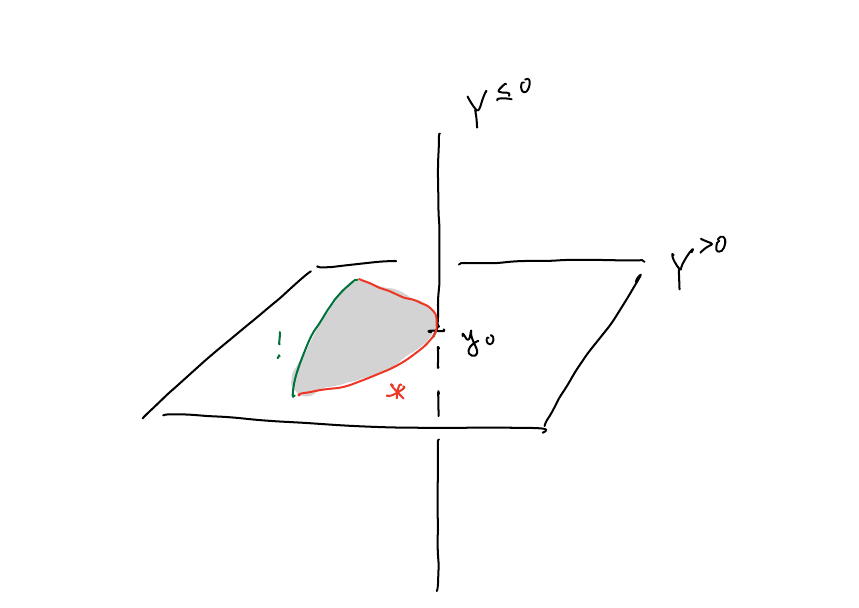}
\caption{Structure of $k_{Y^{>0} \cap B \cap \{\Real f \geq 0\}}$.}\label{Fig2}
\end{figure}

Now the inclusion   $U \subset V$ induces a restriction map $k_V\to k_U$
and hence  a map on corepresented functors  
\beq\label{eq:can map via coreps}
\xymatrix{
\phi_{f, y_0} i^!  \ar[r] &  \phi_{F, y_0}.
}
\eeq

We say  a sheaf $\cF$ on $Y$ is $\CC^{\times}$-monodromic if it is locally constant along each $\CC^{\times}$-orbit.\footnote{In fact, we will only use that a $\CC^{\times}$-monodromic sheaf $\cF$ is locally constant along each $\RR^{>0}$-orbit.}

\begin{theorem}\label{main}
If $\cF \in \Sh_{\Lambda}(Y)$ is $\CC^{\times}$-monodromic (and has singular support in some subanalytic closed conic Lagrangian $\Lambda$) then the map ~\eqref{eq:can map via coreps} induces an equivalence
\beq
\xymatrix{
\phi_{f, y_0} i^! \cF \ar[r]^-\sim &  \phi_{F, y_0} \cF.
}
\eeq
\end{theorem}

\begin{proof}
The kernel of the map~\eqref{eq:can map via coreps} is corepresented by  the $!$-extension
$k_{V \setminus U}$ of the constant sheaf on $V \setminus U$.

Let us assume $B$ is the polyball cut out by
\beq
 d_{>0} < a,\qquad d_0 +  d_{< 0}  < b  
\eeq
where $a>0$ is small compared with $b>0$.
Then we will show that $V \setminus U$ is foliated by flow lines for the action of $\BR^{>0} \subset \CC^{\times}$ entering through its $*$-boundary and exiting through its $!$-boundary.

Note the zero locus of the Euler vector field $v$ generating the $\BR_{>0}$-action is the fixed locus $Y^0$.
Since $F|_{Y^0 \setminus y_0} < 0$, on the closure of $U\setminus V$,
the vector field  $v$ only  vanishes at $y_0$.


Observe that the boundary components of  $V$ are given by the equations
\beq
F = 0, \qquad d_{>0} = a.
\eeq
Indeed,  if  $d_0 +  d_{< 0}  = b$, then by construction  $F(y) = \Real f(\pi(y)) - b$, so that for $F(y) \geq 0$, we must have $f(\pi(y)) \geq b$. But since we have chosen $a>0$  small with respect to $b>0$, we must have that eventually $f(\pi(y)) <b$ for any $y$ with $d_{>0}(y) < a$.

Observe that $k_{V \setminus U}$ is a $*$-extension along  the  $F = 0$ boundary component
and a $!$-extension along the $d_{>0} = a$ boundary component.
Now we further observe how the Euler vector field  $v$ interacts with the boundary components:
\begin{itemize}
 
\item
$v$  is outward pointing on the face  $F = 0$ since in the expression $F(y) = \Real f(\pi(y)) - d_0(y) - d_{<0}(y)$, the terms $f(\pi(y))$ and $-d_{<0}(y)$ both decrease and the term $d_0(y)$ is constant.

\item  $v$ is evidently inward pointing on the face $d_{>0} = a$ since the term $d_{>0}(y)$ decreases.
\end{itemize}

We conclude  $V \setminus U$ is foliated by flow segments for the action of $\BR^{>0} \subset \CC^{\times}$ entering through its $*$-boundary and exiting through its $!$-boundary.  The kernel of $\phi_{f, y_0} i^! \cF \to \phi_{F, y_0} \cF$ is global sections of the sheaf $\uHom(k_{V \setminus U}, \cF)$. If $\cF$ is $\CC^\times$-monodromic, then the pushforward of $\uHom(k_{V \setminus U}, \cF)$ to the quotient $Y/\BR_{>0}$ is already zero, so its global sections are zero.
\end{proof}

\subsection{Calculation of the shift}\label{tip}
As in Section \ref{commuting}, let $Y^{>0}$ be the attracting locus to $y_0$ and let $Y^{\leq 0}$ be the repelling locus.
In applying Theorem~\ref{main}, we want to be able to interpret the vanishing cycles $\phi_{F, y_0}$ as a microstalk.

For this we will use the following general result.

\begin{prop}\label{CleanMicrostalk}
Let $G$ be a function on a complex analytic manifold $X$. Let $\Lambda \subset T^*X$ be a complex subanalytic closed conic Lagrangian. Suppose that the graph of the differential $\Gamma_{dG}$ intersects $\Lambda$ cleanly along smooth points. Then \beq\label{GeneralMicrostalk} \xymatrix{ \phi_{G, x_0}[\ind(\lambda_{\ver}, \lambda, \lambda_G)/2 - \dim(\Lambda \cap \Gamma_{dG})] : \Sh_{\Lambda}(X) \ar[r] & \Vect} \eeq
is exact, calculates microstalk at $dG(x_0)$, and commutes with Verdier duality.\footnote{By dimension we always mean complex dimension.}

Here $\ind(\lambda_{\ver}, \lambda, \lambda_G)$ is the Maslov index of three Lagrangians \beq \label{3Lags}\lambda_{\ver} := T_{d G(x_0)} T^*_{x_0} X, \quad \lambda := T_{d G(x_0)}\Lambda, \quad \lambda_G := T_{d G(x_0)}\Gamma_{dG}\eeq inside the symplectic vector space $T_{dG(x_0)}T^*X$.
\end{prop}
\begin{proof}
The microstalk category $\Sh_{\Lambda}(X)_{dG(x_0)}$ is defined in Section 6.1 of \cite{KS} by localizing with respect to all sheaves singular supported away from $\xi_0$.
It is invariant under contact transformations by Corollary 7.2.2 of \cite{KS}. Since $dG(x_0)$ is a smooth point of $\Lambda$, the proof of Corollary 1.6.4 of \cite{KK} gives a contact transformation taking a neighborhood of $dG(x_0)$ in $\Lambda$ to a neighborhood in the conormal bundle to a smooth hypersurface. Therefore $\Sh_{\Lambda}(X)_{dG(x_0)} \simeq \Vect$ by Proposition 6.6.1 of \cite{KS}.

The functor \eqref{GeneralMicrostalk} factors through $\Sh_{\Lambda}(X)_{dG(x_0)}$. Therefore it just remains to prove that \eqref{GeneralMicrostalk} is exact and factors as the quotient functor to $\Sh_{\Lambda}(X)_{dG(x_0)}$ followed by an equivalence $\Sh_{\Lambda}(X)_{dG(x_0)} \xrightarrow{\sim} \Vect$. By Lemma 7.5.2 and Proposition 7.5.3 of \cite{KS} (the number $d$ in the statement of their lemma does not depend on the function $G$), it suffices to assume that $\Lambda = T^*_ZX$ is the conormal to a complex submanifold. Then the constant sheaf $k_Z[\dim Z]$ supported on $Z$ generates $\Sh_{\Lambda}(X)_{dG(x_0)}$. Let $V \subset Z$ be the image of $\Gamma_{dG} \cap \Lambda$. Then $G|_Z$ is constant along $V$ and approximated by a nondegenerate quadratic form of signature $\ind(\lambda_{\ver}, \lambda, \lambda_G)$ in the directions normal to $V$.  Therefore $\phi_G(k_Z[\dim Z + \ind(\lambda_{\ver}, \lambda, \lambda_G)/2]) \simeq k_V[\dim V]$. Hence \eqref{GeneralMicrostalk} is exact and sends $k_Z[\dim Z]$ to a one-dimensional vector space as desired.
\end{proof}

Below we prove the standard fact that in the complex analytic setting, microstalk at a smooth point of the singular support commutes with Verdier duality.

\begin{prop}\label{MicroVerdier}
If $\Lambda \subset T^*X$ is a complex subanalytic closed conic Lagrangian then microstalk at a smooth point $\xi_0 \in \Lambda$ commutes 
with Verdier duality.
\end{prop}
\begin{proof}
Since $\Lambda$ is complex subanalytic, Verdier duality on $\Sh_{\Lambda}(X)_{\xi_0}$ preserves singular support, see Exercise V.13 of \cite{KS}. Both the microstalk at $\xi_0$ and its conjugate by Verdier duality vanish on sheaves singular supported away from $\xi_0$, so they both factor through the microstalk category $\Sh_{\Lambda}(X)_{\xi_0} \simeq \Vect$ and therefore coincide.
\end{proof}

\begin{remark}[Real versus holomorphic tangent bundles]\label{realcomplex}
For a holomorphic function $f$ on $Y^{>0}$ the differential $df(y) \in (T_y^*Y^{>0} \otimes_{\RR} \CC)^{1, 0}$ lies in the holomorphic tangent bundle. Alternatively, we could take the differential of the real part $(d \Real f) (y) \in T_y^*Y^{>0}$. We identify the holomorphic tangent bundle with the tangent bundle of the underlying real manifold by \beq \xymatrix{\Real:(T_y^*Y^{>0} \otimes_{\RR} \CC)^{1, 0} \ar[r] & T_y^*Y^{>0} \otimes_{\RR} \CC  \ar[r] & T_y^*Y^{>0}}\eeq where the second map is dual to the inclusion $T_yY^{>0} \hookrightarrow T_yY^{>0}\otimes_{\RR} \CC.$ Therefore we let $T^*Y^{>0}$ also denote the holomorphic tangent bundle. For a holomorphic function $f$, we are free to confuse $df(y)$ with $d\Real f(y)$.
Moreover, the two shifted conormal bundles $T^*_{Y^{>0}} Y + df$ and $T^*_{Y^{>0}} Y + d\Real f$ are identified inside $T^*Y$.
\end{remark}

\begin{prop}\label{shift}
Let $\Lambda \subset T^*Y$ be a closed conic complex subanalytic Lagrangian preserved by the hyperbolic action. Let $f: Y^{>0} \rightarrow \CC$ is holomorphic function that is equivariant for $\CC^{\times}$ acting on $\CC$ with weight equal to the smallest weight of the $\CC^{\times}$-action on $Y^{>0}$. Assume also that $T^*_{Y^{>0}}Y + df$ intersects $\Lambda$ cleanly along smooth points. Then there exists an extension $F: Y \rightarrow \BR$ as in \eqref{FDef}, such that the vanishing cycles functor
\beq \label{MicrostalkF}
\xymatrix{
 \phi_{F, y_0}[\dim Y^{\leq 0} - \dim (\Lambda \cap (T_{Y^{>0}}^* Y + df))]: \Sh_{\Lambda}(Y) \ar[r] &  \Vect
 }
 \eeq is exact, calculates microstalk at $df(y_0)$, and commutes with the Verdier duality.
\end{prop}
\begin{proof}
Proposition \ref{MaslovIndex} says that for some choice of coordinates, the construction from Section \ref{commuting} gives an extension $F$ such that $\Gamma_{dF}$ intersects $\Lambda$  cleanly along smooth points. Therefore Proposition \ref{CleanMicrostalk} implies that, up to a shift, $\phi_{F, y_0}$ is exact, calculates microstalk at $df(y_0)$, and commutes with Verdier duality.
Moreover the shift is given by combining \eqref{GeneralMicrostalk} and \eqref{index}.
\end{proof}

If $\Lambda$ is a smooth closed conic Lagrangian then it is the conormal to a smooth submanifold $\Lambda = T^*_ZY$ (\cite{KS} Exercise A.2) and the Maslov index of \eqref{3Lags} is the signature of the Hessian of $F$ restricted to $Z$. This is not the case in our application. Although the global nilpotent cone is smooth near $df(y_0)$, it is not smooth near $y_0$ in the zero section. But the Maslov index only depends on the tangent space to $\Lambda$ at $df(y_0)$. So to compute the index we will choose a submanifold $Z$ whose conormal bundle is tangent to $\Lambda$ at $df(y_0)$.

\begin{prop}\label{ConormalTangentLambda}
Let $\Lambda \subset T^*Y$ be a closed conic Lagrangian that is smooth near $df(y_0)$ and preserved by the hyperbolic $\BG_m$-action. Set $ \lambda = T_{d f(y_0)}\Lambda$. Then there exists a hyperbolic $\BG_m$-stable submanifold $Z \subset Y$ such that the tangent spaces coincide:\beq \label{TangentCoincide} T_{df(y_0)}T_Z^*Y = \lambda.\eeq
\end{prop}
\begin{proof}
If $df(y_0) = 0$ is in the zero section, then by smoothness $\Lambda = T_Z^*Y$ is conormal to a smooth $\BG_m$-equivariant submanifold $Z$, above a neighborhood of $y_0 \in Y$.

If $df(y_0) \neq 0$ is not in the zero section, then the tangent space to the conormal $T_{df(y_0)}T_Z^*Y$ is determined by the tangent space $T_{y_0}Z$ plus the quadratic order behavior of $Z$ in the $df(y_0)$ codirection. Set $\rho = \lambda \cap \lambda_{\ver}$ regarded as a subspace of the vertical $\lambda_{\ver} = T^*_{y_0}Y$. (Note $\rho$ contains at least the tangent to the line through $df(y_0)$.) There are two different things that we could mean by the orthogonal of $\rho$: its symplectic orthogonal $\rho_{\perp} \subset T_{df(y_0)}T^*Y$, or  its orthogonal under duality $\rho^{\perp} \subset T_{y_0}Y$.

Let $\sfL$ be the set of:
\begin{itemize} \label{CotangentSubmanifold}
\item Lagrangians $\lambda' \subset T_{df(y_0)}T^*Y$ with vertical component $\lambda' \cap \lambda_{\ver} = \rho$,
\item or equivalently Lagrangians $\lambda'/\rho \subset \rho_{\perp}/\rho$ transverse to the vertical fiber.
\end{itemize}
Then $\sfL$ is a torsor for the vector space of quadratic forms on $\rho^{\perp}$. Indeed if we choose a reference Lagrangian $\lambda_0 \in \sfL$ then we can identify $\rho_{\perp}/\rho = T^* \rho^{\perp}$ using $\lambda_0/\rho$ as the zero section. A Lagrangian in $T^* \rho^{\perp}$ transverse to the vertical fiber is the graph of a quadratic form on $\rho^{\perp}$.\footnote{The graph of a bilinear form $V \rightarrow V^*$ is Lagrangian if and only if the bilinear form is symmetric.}

Let $\sfG$ be the set of germs of submanifolds $Z' \subset Y$ with tangent space $T_{y_0}Z' = \rho^{\perp}$ up to quadratic order equivalence. This is a torsor for the space of quadratic forms from $\rho^{\perp}$ to the normal space. Indeed if we choose a reference $[Z_0] \in \sfG$ then quadratic germs in $\sfG$ can be identified with the graphs of quadratic forms $T_{y_0}Z_0 \rightarrow (T_{Z_0}Y)_{y_0}$.

Taking the tangent space to the conormal bundle gives a map  \beq \label{TangentToConormal} \xymatrix{\sfG\ar[r] & \sfL, & [Z'] \ar@{|->}[r] & T_{df(y_0)}T^*_{Z'}Y.}\eeq

There are two commuting $\BG_m$ actions on $T^*Y$: the cotangent fiber scaling action and the Hamiltonian action induced by the hyperbolic action on $Y$. Neither action fixes the point $df(y_0)$. However there is some combination of the two actions which fixes $df(y_0)$ and therefore acts linearly on $T_{df(y_0)}T^*Y$. This is because $y_0$ is a fixed point and $f:Y^{>0} \rightarrow \CC$ is $\BG_m$-equivariant where $\BG_m$ acts on $\CC$ with some weight $n$. Therefore we get a $\BG_m$-action on $\mathsf{L}$ for which $\mathsf{G} \rightarrow \mathsf{L}$ is equivariant.\footnote{Alternatively $[df(y_0)] \in P^*Y$ is a fixed point in the projectivized cotangent bundle so $\BG_m$-acts on the set of Legendrians inside $T_{[df(y_0)]}S^*Y$.}

First, choose any $\BG_m$-stable germ $[Z_{0}] \in \sfG$.\footnote{Lifting the $\BG_m$-stable tangent space $T_{y_0}Z_0 = \rho^{\perp}$ to such a quadratic germ is the same argument as the final paragraph of the proof.} This gives identifications $\sfL \simeq \Hom(\Sym^2(T_{y_0}Z_0), \CC)$ and $\sfG \simeq \Hom(\Sym^2(T_{y_0}Z_0), (T_{Y_0}Y)_{y_0})$ compatibly with all $\BG_m$-actions (in particular the weight $n$ action on $\CC$). Then \eqref{TangentToConormal} is identified with \beq\label{TangentToConormaltriv} \xymatrix{\Hom(\Sym^2(T_{y_0}Z_0), (T_{Z_0}Y)_{y_0}) \ar[r] & \Hom(\Sym^2(T_{y_0}Z_0), \CC)}\eeq given by composition with $df:(T_{Z_0}Y)_{y_0} \rightarrow \CC$.

Since $\Lambda \subset T^*Y$ is preserved by both $\GG_m$-actions, its tangent space $\lambda$ is preserved by the combined action on $T_{df(y_0)}T^*Y$. Therefore the associated quadratic form $\Sym^2(T_{y_0}Z_0) \rightarrow \CC$ is $\BG_m$-equivariant. Lift it along \eqref{TangentToConormaltriv} to a $\BG_m$-equivariant quadratic form $\Sym^2(T_{y_0}Z_0) \rightarrow (T_{Z_0}Y)_{y_0}$. This gives a $\BG_m$-equivariant germ $[Z] \in \sfG$ whose conormal is $\lambda$.

It just remains to lift the $\BG_m$-stable germ $[Z] \in \sfG$ to a genuine $\BG_m$-stable submanifold. Suppose $[Z]$ is cut out from $\Spec(\cO_{Y, y_0}/\frakm^3)$ by polynomials $\overline{f}_1, \dots \overline{f}_d \in \cO_{Y, y_0}/\frakm^3$
that are $\BG_m$-eigenvectors and  whose differentials at $y_0$ are linearly independent. Lift them to  polynomials $f_1, \dots f_d$ in $\cO_{Y, y_0}$ that are $\BG_m$-eigenvectors. Then the lifts $f_1, \dots f_d$ cut out a $\BG_m$-stable submanifold $Z$ in a neighborhood of $y_0$ satisfying the desired $T_{df(y_0)}T_Z^*Y = \lambda$.
\end{proof}

Now we are ready to replace $\Lambda$ by a conormal bundle and calculate the shift.
Let $\ind(\lambda_{\ver}, \lambda, \lambda_F)/2$ denote the Maslov index of the three Lagrangians \beq \lambda_{\ver} := T_{d f(y_0)} T^*_{y_0} Y, \quad \lambda := T_{d f(y_0)}\Lambda, \quad \lambda_F := T_{d f(y_0)}\Gamma_{dF}\eeq inside the symplectic vector space $T_{df(y_0)}T^*Y$.
Moreover let $\lambda_f := T_{d f(y_0)}(T^*_{Y^{> 0}}Y + df)$ be tangent to the shifted conormal.

\begin{prop} \label{MaslovIndex}
Under the assumptions of Proposition \ref{shift}, the function $F$ can be chosen such that $\Gamma_{dF}$ intersects $\Lambda$ cleanly along smooth points and \beq \label{index} \ind(\lambda_{\ver}, \lambda, \lambda_F)/2 - \dim(\lambda \cap \lambda_F) = \dim Y^{\leq 0} - \dim(\lambda \cap \lambda_f).\eeq
\end{prop}

\begin{proof}
Use Proposition \ref{ConormalTangentLambda} to choose a $\CC^{\times}$-stable submanifold $Z$ containing $y_0$ such that $T_{df(y_0)}T^*_ZY = \lambda$. Let $Z^{> 0}$ be the attracting locus to $y_0$ and $Z^{\geq 0}$ be the repelling locus.
There is a short exact sequence \beq\label{ProjectBase} 0 \rightarrow \lambda_{\ver} \rightarrow T_{df(y_0)} T^*Y \xrightarrow{\pi} T_{y_0} Y \rightarrow 0.\eeq

By \eqref{ProjectBase} we have \beq \label{VertHorzf} \dim(\lambda \cap \lambda_f) = \dim(\lambda \cap \lambda_f \cap \lambda_{\ver}) + \dim (\pi(\lambda \cap \lambda_f)).\eeq
Since $\pi(\lambda_f) \subset T_{y_0}Y^{>0}$, the restriction of $\pi$ to $\lambda \cap \lambda_f$ factors through \beq \lambda \cap \lambda_f \rightarrow T_{df(y_0)} T^*Y|_{Y^{> 0}} \xrightarrow{\rho} T_{df(y_0)} T^*Y^{> 0} \rightarrow T_{y_0} Y^{> 0}.\eeq Note that
\beq\label{Rhof} \rho(\lambda \cap \lambda_f) = T_{df(y_0)} T^*_{Z^{>0}} Y^{>0} \cap T_{df(y_0)} \Gamma_{df}.\eeq
Moreover
\beq\label{DimfVert} \dim(\lambda \cap \lambda_f \cap \lambda_{\ver}) = \dim Y^{\leq 0} - \dim Z^{\leq 0}.\eeq
Combining the above gives
\beq \label{Dimf} \dim(\lambda \cap \lambda_f) = \dim Y^{\leq 0} - \dim Z^{\leq 0} + \dim(T_{df(y_0)} T^*_{Z^{>0}} Y^{>0} \cap T_{df(y_0)} \Gamma_{df}).\eeq

A similar argument shows that \beq \dim(\Lambda \cap (T^*_{Y^{>0}} Y + df)) = \dim Y^{\leq 0} - \dim Z^{\leq 0} + \dim(T^*_{Z^{>0}} Y^{>0} \cap \Gamma_{df}).\eeq
Since the intersection $\Lambda \cap (T^*_{Y^{>0}} Y + df)$ was assumed to be clean, we find that the intersection $T^*_{Z^{>0}} Y^{>0} \cap \Gamma_{df}$ is also clean.

Note that $\Gamma_{dF}|_{Y^{>0}}$ is a section of the shifted conormal $T^*_{Y^{>0}}Y + df \rightarrow Y^{> 0}$, that is determined by the choice of coordinates in Section \ref{commuting}. Now let $\CC^{\times}$ act on $T^*Y$ fixing $df(y_0)$ by a combination of the hyperbolic action and scaling of the cotangent fibers. Choosing coordinates appropriately we can make $\Gamma_{dF}|_{Y^{>0}}$ any $\CC^{\times}$-stable section of $T^*_{Y^{>0}}Y + df$. Note that $\Lambda \cap (T^*_{Y^{>0}}Y + df)$ is $\CC^{\times}$-stable and it is a submanifold by the cleanness assumption. Since its image under $T^*Y|_{Y^{>0}} \rightarrow T^*Y^{>0}$ is $T^*_{Z^{>0}} Y^{>0} \cap \Gamma_{df}$, we can choose coordinates so that \beq\label{LambdaFIntersect} \dim(\Lambda \cap \Gamma_{dF}) = \dim(T^*_{Z^{>0}} Y^{>0} \cap \Gamma_{df}).\eeq

Since $F$ is a nondegenerate quadratic form in the repelling directions we have $\pi(\lambda \cap \lambda_F) \subset T_{y_0} Y^{> 0}$.
Therefore \beq\label{InclusionIntersect} \rho(\lambda \cap \lambda_F) \subset \rho(\lambda) \cap \rho(\lambda_F) = T_{df(y_0)} T^*_{Z^{>0}} Y^{>0} \cap T_{df(y_0)} \Gamma_{df}\eeq
implies that \beq \dim(\lambda \cap \lambda_F) = \dim(\rho(\lambda \cap \lambda_F)) \leq \dim(T_{df(y_0)} T^*_{Z^{>0}} Y^{>0} \cap T_{df(y_0)} \Gamma_{df}) = \dim(\Lambda \cap \Gamma_{dF})\eeq
Therefore the inequality is actually an equality \beq\label{DimF} \dim(\lambda \cap \lambda_F) = \dim(T_{df(y_0)} T^*_{Z^{>0}} Y^{>0} \cap T_{df(y_0)} \Gamma_{df})\eeq
and the intersection $\Lambda \cap \Gamma_{dF}$ is clean.

Recall that $f$ is $\CC^{\times}$-equivariant, where $\CC^{\times}$ acts linearly on the target $\CC$ with weight equal to the smallest weight of the $\CC^{\times}$ action on $Y^{>0}$. Therefore 
\beq \dim(T^*_{Z^{>0}} Y^{>0} \cap \Gamma_{df}) = \dim(Z^{>0}).\eeq
Hence \eqref{LambdaFIntersect} and cleaness of the intersection $\Lambda \cap \Gamma_{dF}$ imply that $0$ has multiplicity $2\dim(Z^{>0})$ as an eigenvalue of the Hessian of $F|_Z$. Since $F|_{Z^{\leq 0}}$ is negative definite, the Maslov index is \beq\label{HessianIndex} \ind(\lambda_{\ver}, \lambda, \lambda_F) = 2\dim Z^{\leq 0},\eeq
the negative signature of the Hessian of $F|_Z$.
Combining \eqref{Dimf}, \eqref{DimF},  and \eqref{HessianIndex} implies the desired \eqref{index}.
\end{proof}

\begin{ex}
Let $Y = \AA^2$ with hyperbolic action $z \cdot (x, y) = (zx, z^{-1} y)$. So $i$ is the inclusion of $Y^{>0} = \{y = 0\}$.

The Whittaker functional of the skyscraper sheaf $k_0$ is \beq \phi_{x, 0} i^! k_{0} = \phi_{x, 0} k_0 = k.\eeq The singular support $\Lambda = T_0^* \AA^2$ of the skyscraper intersects the shifted conormal bundle $
(T_{Y^{>0}} \AA^2 + dx)$ in the one dimensional space $(T_{y = 0}^* \AA^2)_0 + dx$ so \beq \dim Y^{\leq 0} - \dim(T_0^* \AA^2 \cap (T_{Y^{>0}}^* \AA^2 + dx)) = 0\eeq and there is no shift.

The Whittaker functional of the perverse sheaf $k_{x = 0}[1]$ is \beq \phi_{x, 0} i^! k_{x = 0}[1] = \phi_{x, 0} k_{x = 0}[-1] = k[-1].\eeq The resulting vector space becomes perverse after shifting by \beq \dim Y^{\leq 0} - \dim(T^*_{x = 0} \AA^2 \cap (T_{Y^{>0}}^* \AA^2 + dx)) = 1\eeq because the singular support $\Lambda = T^*_{x = 0} \AA^2$ intersects the shifted conormal bundle in a single point.
\end{ex}

\section{Application to automorphic sheaves}

Let  $X$ be a smooth connected projective complex curve with canonical bundle denoted by $\omega$.
Let $G$ be a complex reductive group with Borel subgroup $B\subset G$ with unipotent radical $N = [B,B]$ and universal Cartan $T=B/N$. For concreteness, we will fix a splitting $T\subset B$.

\subsection{The Whittaker functional under uniformization}\label{whit}
Let $\rho^{\vee}$ be half the sum of the positive coroots. Choose a square root $\omega^{1/2}$ of the canonical bundle and consider the $T$-bundle $\rho^{\vee}(\omega) := 2\rho^{\vee}(\omega^{1/2})$. Its key property is that for every simple root $\alpha$, the associated line bundle $\rho^{\vee}(\omega) \times_T \CC_{\alpha} = \omega^{\langle \rho^{\vee}, \alpha \rangle} = \omega$ is canonical.
Let $\Bun_N^{\omega}(X)$ be the moduli of $B$-bundles on $X$ whose underlying $T$-bundle is $\rho^{\vee}(\omega)$.  Thus $\Bun_N^{\omega}(X)$ classifies maps $X\to \pt/B$ such that the composition with $\pt/B\to \pt/T$ classifies the $T$-bundle $\rho^{\vee}(\omega)$. Such maps factor through the classifying space of  $B \times_T 2\rho^{\vee}(\BG_m) \simeq N \rtimes \BG_m$ as in the diagram:
\beq
\xymatrix{
X \ar@/_/[ddr]_{\omega^{1/2}} \ar@/^/[drr] \ar[dr]\\
&\pt/(N \rtimes \BG_m) \ar[d] \ar[r] & \pt/B \ar[d] \\
&\pt/\BG_m \ar[r]_{2\rho^{\vee}} & \pt/T}\eeq
 So alternatively $\Bun_N^{\omega}(X)$ is represented by maps to $\pt/(N \rtimes \BG_m)$ such that the composition to $\pt/\BG_m$ classifies $\omega^{1/2}$.

The semidirect product $N \rtimes \BG_m$ mentioned above is formed by letting $z \in \BG_m$ act on $N$ by conjugation by $2\rho^{\vee}(z) \in T$. In other words, $N \rtimes \BG_m \subset B$ is the subgroup of the Borel generated by $N$ and $2\rho^{\vee}(\BG_m)$. Consider the action of $N \rtimes \BG_m$ on $\frn^*$ by \beq \label{ScaleAdjointAction} z \cdot X = \Ad_{2\rho^{\vee}(z)}(z^2 X),\eeq the product of scaling and the adjoint $T$-action.
\begin{prop}
The cotangent bundle $T^*\Bun_N^{\omega}(X)$ is represented by maps $X \to \frn^*/(N \rtimes \BG_m)$, where we quotient by the \eqref{ScaleAdjointAction} action, such that the composition $X \to \pt/\BG_m$ classifies the line bundle $\omega^{1/2}$.
\end{prop}
\begin{proof}
By definition $\Bun_N^{\omega}(X)$ is a fiber of the smooth (but not representable) map $p:\Bun_B(X) \to \Bun_T(X)$. The relative tangent complex of $p$ is $( \pi_* \frb_{F^{\univ}}[1]\to \pi_*\frt_{F^{\univ}}) \simeq \pi_* \frn_{F^{\univ}}[1]$ given by pushing forward vector bundles associated to the universal $B$-bundle $F^{\univ}$ along $\pi: X \times \Bun_B(X) \to \Bun_B(X)$. The tangent complex of $\Bun_N^{\omega}(X)$ is the restriction $\pi_* \frn_{F^{\univ}}|_{\Bun_N^{\omega}(X)}[1]$. Taking the stalk at a point $F \in \Bun_N^{\omega}(X)$ gives the tangent space $T_F \Bun_N^{\omega}(X) = H^1(\frn_F)$. By Serre duality the cotangent space is
\beq T^*_F\Bun_N^{\omega}(X) = H^0(\frn^*_F \otimes \omega).\eeq
Here $\frn^*_F$ is the vector bundle obtained from $F$ via the adjoint action of $B$ on $\frn^*$. Whereas $\frn^*_F \otimes \omega$ is obtained from $F$ via $N \rtimes \BG_m$ acting on $\frn^*$ by \eqref{ScaleAdjointAction}.

So giving a cotangent vector in $T_F^* \Bun_N^{\omega}(X)$ is equivalent to lifting the classifying map $X \to \pt/(N \rtimes \BG_m)$ of the bundle $F$ to a map $X \to \frn^*/(N \rtimes \BG_m)$. It was important that we modified the adjoint $B$-action on $\frn^*$ by also scaling so as to incorporate the canonical twist from Serre duality.
\end{proof}

Let $f: \Bun_N^{\omega}(X) \to \AA^1$ be the function given by the sum of the functions
\beq
\xymatrix{
\Bun_N^{\omega}(X) \ar[r] &  \Bun_{\BG_a}^{\omega}(X) \simeq H^1(X, \omega) \times BH^0(X, \omega) \ar[r] & H^1(X, \omega)\simeq \AA^1
}\eeq
induced by projection onto each simple root space $N \to N/[N,N] \to \BG_a$.
The graph of its differential $\Gamma_{df} \subset T^* \Bun_N^{\omega}(X)$ is represented by \beq \psi/(N \rtimes \BG_m) \subset \frn^*/(N \rtimes \BG_m)\eeq where
$\psi: \frn \to \frn/[\frn, \frn] \to \AA^1$
is given by summing over the simple root spaces.
To see that the expression $\psi/(N \rtimes \BG_m)$ makes sense we need to check that $\psi$ is invariant under the $(N \rtimes \BG_m)$-action. Indeed $\psi$ factors through the abelianization so it is $N$-invariant. Furthermore the adjoint action of $2\rho^{\vee}(z)$ scales the $\alpha_i$ simple root space component of $\psi$ by $z^{-\langle 2\rho^{\vee}, \alpha_i\rangle} = z^{-2}$ cancelling out the $\BG_m$-scaling action.

Let $\Bun_G(X)$ be the moduli of $G$-bundles on $X$. Recall
the cotangent bundle of $\Bun_G(X)$ is the moduli of Higgs bundles
\beq
T^* \Bun_G(X) \simeq \{E, \sigma \in H^0(\frg^*_E \otimes \omega)\}
\eeq
classifying maps $X\to \frg^*/(G\times \BG_m)$ such that the composition to $\pt/\BG_m$ classifies the line bundle $\omega$.
The global nilpotent cone $\Lambda \subset T^*\Bun_G(X)$ is the moduli of everywhere nilpotent Higgs bundles
\beq
\Lambda = \{E, \sigma \in H^0(\cN^*_E \otimes \omega)\}
\eeq
classifying maps  $X\to \cN^*/(G\times \BG_m)$
such that the composition to $\pt/\BG_m$ classifies the line bundle $\omega$.

The Whittaker functional
\beq
\xymatrix{
\phi_{f, \rho^{\vee}(\omega)} i^!: \Sh_{\Lambda}(\Bun_G(X)) \ar[r] & \Vect
}
\eeq
is  $!$-pullback along the natural  induction map
\beq
\xymatrix{
i : \Bun_N^{\omega}(X) \ar[r] & \Bun_G(X)
}\eeq followed  by vanishing cycles for $f$ at the point $\rho^{\vee}(\omega)$.
Note one could alternatively take global sections rather than stalk of the vanishing cycles $\phi_{f}$, but this will give the same result by the contraction principle (\cite{KS} Proposition 3.7.5).

To apply our general results, we would like to locally uniformize the moduli in play and replace them by smooth schemes.
To this end, fix a closed point $x \in X$.  

First, by taking $n$ large enough, we may factor $i$ through a closed embedding followed by a smooth projection
\beq
\xymatrix{
i:\Bun_N^{\omega}(X) \ar[r] &  \Bun^{\omega}_{G, N}(X, nx)  \ar[r] &  \Bun_G(X).
}
\eeq
Here $\Bun^{\omega}_{G, N}(X, nx)$ is the moduli space of $G$-bundles on $X$ with a $B$-reduction on the $n$th order neighborhood $D_n(x)$  whose underlying $T$-bundle is $\rho^{\vee}(\omega)|_{D_n(x)}$. The maps factoring $i$ are the natural induction maps; the second is clearly a smooth projection, and we will see momentarily that the first is a closed embedding.

Next,  introduce the moduli $\Bun_{G}^{\omega}(X, nx)$ classifying
 $G$-bundles on $X$ with a reduction on the $n$th order neighborhood $D_n(x)$  to the $\GG_m$-bundle $\om^{1/2}$
 via  the inclusion $2\rho^{\vee}:\GG_m \to T \subset G$.\footnote{Choosing a trivialization of $\om^{1/2}$ over  $D_n(x)$
  gives an isomorphism $\Bun_{G}(X, nx) \simeq \Bun_{G}^{\omega}(X, nx)$
  where $\Bun_{G}(X, nx)$ classifies
 $G$-bundles on $X$ with a trivialization over $D_n(x)$ .}
Form the following induction diagram with a Cartesian square.
\beq
\begin{tikzcd}
\Bun_N^{\omega}(X, nx) \arrow[r, "i'"] \arrow[d] & \Bun_{G}^{\omega}(X, nx) \arrow[d] & \\
i:\Bun_N^{\omega}(X) \arrow[r] & \Bun^{\omega}_{G, N}(X, nx) \arrow[r] & \Bun_G(X)
\end{tikzcd}
\eeq
Thus $\Bun_N^{\omega}(X, nx)$ classifies objects of $\Bun_N^{\omega}(X)$ with a reduction
on  $D_n(x)$  to the $\GG_m$-bundle $\om^{1/2}$
 via  the inclusion $2\rho^{\vee}:\GG_m \to T \subset B$.

Take a quasi-compact open substack $U \subset \Bun_G(X)$ containing the image of $\Bun_N^{\omega}(X)$. Then for $n$ sufficiently large, $\Bun_G^{\omega}(X, nx)|_U$ is a scheme.
Futhermore, at $F \in \Bun_N^{\omega}(X, nx)$, for $n$ sufficiently large, the codifferential
\beq
\xymatrix{
(di')^*:H^0(\frg^*_F \otimes \omega(nx)) \ar@{->>}[r] &  H^0(\frn^*_F \otimes \omega(nx))
}
\eeq
is surjective since $H^1(\frn^{\perp}_F \otimes \omega(nx)) = 0$. Moreover, we can choose $n \gg 0 $ once and for all uniformly over $\Bun_N^{\omega}(X)$ by quasi-compactness.

Thus for $n$ sufficiently large, since $i'$ is a map between smooth schemes with surjective codifferential, it is locally a closed embedding. Applying contraction for the natural $\GG_m$-action considered below, we see $i'$ is in fact a closed embedding. Also, $\Bun_N^{\omega}(X) \to \Bun^{\omega}_{G, N}(X, nx)$ is  a closed embedding  because $i'$ is a base-change of it via a surjective map.

The cotangent bundle $T^* \Bun_G^{\omega}(X, nx)$ classifies data
\beq
T^* \Bun_G^{\omega}(X, nx) = \{E, E|_{D_n(x)} \simeq G \times_T \rho^{\vee}(\omega)|_{D_n(x)}, \sigma \in H^0(\frg^*_E \otimes \omega(nx))\}.
\eeq
Singular support behaves well under smooth pullback. So if $\cF$ is a sheaf on $\Bun_G(X)$ with singular support in the nilpotent cone
\beq
\Lambda = \{E, \sigma \in H^0(\cN^*_E \otimes \omega) \}
\eeq
then the singular support of its smooth pullback to $\Bun_G^{\omega}(X, nx)$ lies in
\beq
\Lambda' = \Lambda \times_{\Bun_G(X)} \Bun_G^{\omega}(X, nx) = \{E, E|_{D_n(x)} \simeq G \times_T \rho^{\vee}(\omega)|_{D_n(x)},  \sigma \in H^0(\cN^*_E \otimes \omega) \}.
\eeq

\subsection{Hyperbolic symmetry}\label{hyp}

To apply Theorem~\ref{main}, we seek a $\BG_m$-action for which $i':\Bun_N^{\omega}(X, nx) \to \Bun_G^{\omega}(X, nx)$ is the attracting locus for the bundle $\rho^{\vee}(\omega)$. Let $f' : \Bun_N^{\omega}(X, nx) \to \AA^1$ be the pullback of $f$ to the uniformized moduli space. To apply Theorem~\ref{main} we also need $f'$ to be $\BG_m$-equivariant for some $\BG_m$ action on $\AA^1$ (which will turn out to have weight 2).

An automorphism $\alpha \in \Aut(G)$ induces an automorphism of $\Bun_G(X)$ by twisting the $G$-actions on the underlying bundles. A $G$-bundle $E$ goes to the $G$-bundle $\alpha E$ with the same total space but the old action of $g$ on $E$ is replaced by the new action of $\alpha(g)$ on $\alpha E$.\footnote{If $E$ is trivialized by $U \rightarrow X$ and described by gluing data $\phi \in H^0(U \times_X U, G)$ then $\alpha E$ is described by the cocycle $\alpha \circ \phi$.}
If the automorphism of $G$ is inner, say it is given by conjugation by $h \in G$, then the action on $\Bun_G(X)$ is entirely stacky in the sense that it is trivial on the set of isomorphism classes of points. Indeed the multiplication by $h$ map $h:E \rightarrow \alpha E$ intertwines the original action with the twisted one.

Suppose now that the automorphism $\alpha$ of $G$ is trivial on $T$. Then $\alpha$ also induces an automorphism of $\Bun_G^{\omega}(X, nx)$ with level structure. A $G$-bundle $E$ with canonically twisted trivialization $\phi: E|_{D_n(X)} \xrightarrow{\sim} G \times_T \rho^{\vee}(\omega)|_{D_n(x)}$ goes to the $G$-bundle $\alpha E$ with trivialization
\beq\xymatrix{
 \alpha E|_{D_n(X)} \simeq E|_{D_n(x)} \ar[r]^-{\phi}  & G \times_T \rho^{\vee}(\omega)|_{D_n(x)}  \ar[r]^-{\alpha} &  G \times_T \rho^{\vee}(\omega)|_{D_n(x)}.
 }
 \eeq The final map is trivial on the $\rho^{\vee}(\omega)$ factor and is well defined because we assumed that the automorphism $\alpha: G \rightarrow G$ is right $T$-invariant. For example if $\alpha(g) = hgh^{-1}$ is an inner automorphism and $h \in T$, then we get an automorphism of $\Bun_G^{\omega}(X, nx)$ that preserves the underlying bundle and changes the trivialization by conjugation by $h$. In other words $\alpha$ acts along the fibers of $\Bun_G^{\omega}(X, nx) \rightarrow \Bun_G(X)$.

\begin{remark} For simplicity ignore the canonical twist and suppose that $G$ is semisimple so we have one point uniformization, \beq \Bun_G(X, nx) = K_n \setminus G(K_x) / G(X - x).\eeq Here $K_n \subset G(\mathcal{O}_x)$ consists of matrices that are the identity to $n$th order. Then the inner automorphism $\alpha(g) = h g h^{-1}$ sends a double coset $K_n g G(X - x)$ to $K_n h g h^{-1} G(X - x) \simeq K_n h g G(X - x)$. Since $h^{-1}$ is a constant function we could absorb it into $G(X - x)$, so alternatively the action is given by changing the trivialization by left multiplication.
\end{remark}

Let $z \in \BG_m$ act on $G$ and $B$ by conjugation by $\rho^{\vee}(z) \in T/Z(G)$ in the torus of the adjoint group. This gives a $\BG_m$-action on the moduli spaces of bundles for which the natural maps between moduli spaces are equivariant.

\begin{prop}[4.7 of \cite{DG}]
Restrict to the connected component of $\Bun_B(X, nx)$ and $\Bun_T(X, nx)$  indexed by the coweight $(2g-2)\rho^{\vee}$.
Then for $n$ sufficiently large, $\Bun_B(X, nx)$ is the attracting locus to $\Bun_T(X, nx)$ in an open neighborhood of $\Bun_T(X, nx)$ inside $\Bun_G^{\omega}(X, nx)|_U$.
\end{prop}
\begin{proof}
The $\BG_m$-action contracts $\Bun_B(X, nx)$ to $\Bun_T(X, nx)$ because $\rho^{\vee}(\BG_m)$ contracts $B$ to $T$. Indeed if $F$ is a $B$-bundle then acting by $\rho^{\vee}(z)$ gives a bundle with the same total space but $b$ acting by $\rho^{\vee}(z)b\rho^{\vee}(z)^{-1}$. As $z \rightarrow 0$, the conjugate $\rho^{\vee}(z)b\rho^{\vee}(z)^{-1}$ approaches an element of $T$ so the $B$-bundle approaches one induced from a $T$-bundle.

It remains to check $\Bun_B(X, nx)$ is the full attracting locus in an open neighborhood of $\Bun_T(X, nx)$ inside $\Bun_G^{\omega}(X, nx)$. This is because $p:\Bun_B(X, nx) \to \Bun_G^{\omega}(X, nx)|_U$ is a closed embedding (we are implicitly restricting to the connected component containing $\rho^{\vee}(\omega)$ and choosing $n$  large) so $p$ is a closed embedding into the attracting locus. Since $\Bun_B(X, nx)$ is smooth, it suffices to show that $p$ is also an open embedding into a neighborhood of the attracting locus about $\Bun_T(X, nx)$. This follows because the derivative over $L \in \Bun_T(X, nx)$, given by the natural map
\beq
\xymatrix{
T_L\Bun_B(X, nx) \simeq H^1(\frb_L(-nx)) \ar[r] &   T_L\Bun_G^{\omega}(X, nx) \simeq H^1(\frg_L(-nx)),
}\eeq
maps isomorphically into the non-negative $\BG_m$-weight spaces.
\end{proof}

Since $p: \Bun_B^{\omega}(X, nx) \rightarrow \Bun_T^{\omega}(X, nx)$ is $\BG_m$-equivariant, the fiber $\Bun_N^{\omega}(X, nx)$ also admits a $\BG_m$-action. But the action on $\Bun_N^{\omega}(X, nx)$ changes the bundles not just the trivializations because conjugation by $\rho^{\vee}(\BG_m)$ is an outer automorphism of $N$.

\begin{prop}
The function $f: \Bun_N^{\omega}(X) \rightarrow \AA^1$ is $\BG_m$-equivariant.
\end{prop}
\begin{proof}
For each positive simple root, projection onto that root space $N \rightarrow \BG_a$ is $\BG_m$-equivariant for the $\rho^{\vee}$ action on $N$ and the scaling action on $\BG_a$. Under uniformization \beq \Bun_{\BG_a}^{\omega}(X) = \mathcal{O}_x dt \setminus K_xdt / \omega(X - x) = H^1(X, \omega) \times BH^0(X, \omega)\eeq the scaling action on $\BG_a$ induces an action that scales the gluing data in $K_xdt$. The residue map $K_x dt \rightarrow \AA^1$ descends to the map $\Bun_{\BG_a}^{\omega}(X) \rightarrow \AA^1$ which is $\BG_m$-equivariant. Since $f$ is defined as the sum over positive simple roots of \beq
\xymatrix{
\Bun_N^{\omega}(X) \ar[r] &  \Bun_{\BG_a}^{\omega}(X)  \ar[r] &  \AA^1,
}\eeq it is $\BG_m$-equivariant.
\end{proof}

We are interested in sheaves on $\Bun_G^{\omega}(X, nx)$ pulled back from $\Bun_G(X)$ so they will certainly be $\BG_m$-equivariant. Alternatively, having singular support in $\Lambda' \subset T^*\Bun_G^{\omega}(X, nx)$, implies constructibility along the orbits of this $\BG_m$-action.

Applying theorem \ref{main} for $Y = \Bun_G^{\omega}(X, nx)$, $y_0 = \rho^{\vee}(\omega)$ with its canonical level structure, $Y^{> 0} = \Bun_N^{\omega}(X, nx)$, and $Y^{\leq 0} = \Bun_{B^-}(X, nx)$ gives the following.
Let $f'$ be the pullback to $\Bun_N^{\omega}(X, nx)$ of $f$ .

\begin{prop}
There is an isomorphism of functors
\beq
\phi_{f', \rho^{\vee}(\omega)} i'^! = \phi_{F, \rho^{\vee}(\omega)}: \Sh_{\Lambda'}(\Bun_G^{\omega}(X, nx)) \to \Vect.
\eeq
Here $F$ is a real valued extension of $f'$ as in \eqref{FDef}.
\end{prop}

Vanishing cycles commutes with smooth pullback, so $\phi_{f, \rho^{\vee}(\omega)} i^!$ and $\phi_{f', \rho^{\vee}(\omega)}i'^!$ agree up to a shift
\beq \label{VanishingShiftDifference}
\phi_{f, \rho^{\vee}(\omega)}i^! [2n\dim N] \simeq \phi_{f', \rho^{\vee}(\omega)} i'^! \pi^!.
\eeq
Therefore we are free to pull everything back to $\Bun_G^{\omega}(X, nx)$ where $F$ is defined. Note that $!$-pullback along
\beq
\xymatrix{
 \pi:\Bun_G^{\omega}(X, nx) \ar[r] &  \Bun_G(X)
 } \eeq
is not exact, but by smoothness $\pi^![-n \dim G]$ is.

\subsection{Microstalk along the Kostant section}\label{kost}
Now we will explain how the shifted conormal is the Kostant section of the Hitchin fibration and therefore intersects the global nilpotent cone transversely in a single smooth point.

\begin{prop}\label{KostNilp}
Inside $T^* \Bun_G(X)$ the shifted conormal bundle $T^*_{\Bun_N^{\omega}(X)}\Bun_G(X) + df$\footnote{Here $T^*_{\Bun_N^{\omega}}\Bun_G + df$ consists of  points in $T^* \Bun_G(X) \times_{\Bun_G(X)} \Bun_N^{\omega}(X)$ that under the codifferential of $\Bun_N^{\omega}(X) \rightarrow \Bun_G(X)$ land in the graph $\Gamma_{df} \subset T^* \Bun_N^{\omega}(X)$. Calling it the shifted conormal is a little misleading because $\Bun_N^{\omega}(X) \rightarrow \Bun_G(X)$ is not a closed embedding.} intersects the global nilpotent cone $\Lambda$ transversely at a smooth point.
\end{prop}
\begin{proof}
The shifted conormal bundle consists of cotangent vectors $X \to \frg^*/(N \rtimes \BG_m)$ in $(T^*\Bun_G(X))|_{\Bun_N^{\omega}(X)}$ such that the composition \beq
\xymatrix{
X \ar[r] &  \frg^*/(N \rtimes \BG_m) \ar[r] &  \frn^*/(N \rtimes \BG_m)
}
\eeq lands in $\psi/(N \rtimes \BG_m)$. Therefore $T^*_{\Bun_N^{\omega}(X)}\Bun_G + df$ is represented by the Kostant section
\beq\xymatrix{
(\psi + \frn^{\perp})/(N \rtimes \BG_m) \ar[r] &  \frg^*/(G \times \BG_m).}\eeq
We used the homomorphism
\beq
N \rtimes \BG_m \to G \times \BG_m, \qquad n \rtimes z \mapsto (n \cdot 2\rho^{\vee}(z), z^2).
\eeq
This is a section of the characteristic polynomial map
$\frg^*/(G \times \BG_m) \to \frc^* / \BG_m$
which represents the Hitchin map \beq\xymatrix{ h: T^* \Bun_G(X) \ar[r] &  \frc^*_G(X).}\eeq

Let $T^* \Bun_G(X)^{\reg}$ be the regular locus, represented by $\frg^{*, \reg}/(G \times \BG_m) \subset \frg^*/(G \times \BG_m)$. It is an open substack of $T^* \Bun_G(X)$ because $\frg^{*, \reg} \subset \frg^*$ is open and $X$ is proper. The Hitchin fibration $h$ is smooth after restricting to this regular locus. Since $\psi + \frn^{\perp} \subset \frg^{*, \reg}$ consists of regular elements the Kostant section is contained in $T^* \Bun_G(X)^{\reg}$. After restricting to $T^* \Bun_G(X)^{\reg}$, the Kostant section is a section of a smooth projection so intersects every fiber, in particular the global nilpotent cone, transversely at a smooth point.
\end{proof}

Since $i$ is not a closed embedding we factored it through $\Bun_N^{\omega}(X) \rightarrow \Bun_{G, N}^{\omega}(X, nx)$.

\begin{prop}\label{KostNilp2}
Inside $T^* \Bun_{G, N}^{\omega}(X, nx)$, the shifted conormal bundle \beq \Lambda_f'' := df + T^*_{\Bun_N^{\omega}(X)} \Bun^{\omega}_{G, N}(X, nx)\eeq intersects the global nilpotent cone
\beq \Lambda'' = \Lambda \times_{\Bun_G(X)} \Bun^{\omega}_{G, N}(X, nx)\eeq
transversely at a single smooth point.
\end{prop}
\begin{proof}
The nilpotent cone $\Lambda''$ is contained inside
\beq \mu^{-1}(0)/N := T^* \Bun_G(X) \times_{\Bun_G(X)} \Bun^{\omega}_{G, N}(X, nx). \footnote{The notation $\mu^{-1}(0)/N$ can just be regarded as shorthand, but here is an explanation. The $G(D_n(x))$ action on $T^*\Bun_G^{\omega}(X, nx)$ has a moment map $\mu: T^*\Bun_G^{\omega}(X, nx) \rightarrow \frg^*(D_n(x))$. And $T^* \Bun_{G, N}^{\omega}(X) = \mu^{-1}(\frn^{\perp}(D_n(x)))/N(D_n(x))$ can be described by Hamiltonian reduction for the $N(D_n(x))$-action. So $\mu^{-1}(0)/N(D_n(x))$ is the closed subspace of $T^* \Bun_{G, N}^{\omega}(X, nx)$ where we impose the further condition that the moment map lands in $0 \in \frn^{\perp}(D_n(x))$.} \eeq
Whereas we claim that the shifted conormal $\Lambda_f''$  intersects  $\mu^{-1}(0)/N$ transversely. By the previous Proposition \ref{KostNilp}, \beq \Lambda_f'' \cap (\mu^{-1}(0)/N) = df + T^*_{\Bun_N^{\omega}(X)} \Bun_G(X)\eeq intersects $\Lambda''$ transversely inside $\mu^{-1}(0)/N$ at a single smooth point of $\Lambda''$. Therefore $\Lambda_f''$ and $\Lambda''$ intersect transversely as desired.

To see the claim that $\Lambda_f''$ and $\mu^{-1}(0)/N$ intersect transversely we need to check that their tangent spaces at $df(\rho^{\vee}(\omega))$ together span the whole $T_{df(\rho^{\vee}(\omega))} T^* \Bun_{G, N}^{\omega}(X, nx)$. Projecting onto the horizontal directions, $T_{df(\rho^{\vee}(\omega))} T^* \Bun_{G, N}^{\omega}(X, nx)$ fits into a short exact sequence \beq 0 \rightarrow T^*_{\rho^{\vee}(\omega)} \Bun_{G, N}^{\omega}(X, nx) \rightarrow T_{df(\rho^{\vee}(\omega))} T^* \Bun_{G, N}^{\omega}(X, nx) \rightarrow T_{\rho^{\vee}(\omega)} \Bun_{G, N}^{\omega}(X, nx) \rightarrow 0.\eeq The tangent space to $\mu^{-1}(0)/N$ surjects onto $T_{\rho^{\vee}(\omega)} \Bun_{G, N}^{\omega}(X, nx)$ so it suffices to show that the tangent spaces to $\Lambda_f''$ and $\mu^{-1}(0)/N$ intersected with the vertical subspace $T^*_{\rho^{\vee}(\omega)} \Bun_{G, N}^{\omega}(X, nx)$ together span the whole $T^*_{\rho^{\vee}(\omega)} \Bun_{G, N}^{\omega}(X, nx)$.
The vertical part of the tangent space to $\Lambda_f''$ is the conormal space $(T^*_{\Bun_N^{\omega}(X)} \Bun_{G, N}^{\omega}(X, nx))_{\rho^{\vee}(\omega)}$, which is by definition the kernel in a short exact sequence \beq 0 \rightarrow (T^*_{\Bun_N^{\omega}(X)} \Bun_{G, N}^{\omega}(X, nx))_{\rho^{\vee}(\omega)} \rightarrow T^*_{\rho^{\vee}(\omega)} \Bun_{G, N}^{\omega}(X, nx) \rightarrow T^*_{\rho^{\vee}(\omega)} \Bun_N^{\omega}(X) \rightarrow 0.\eeq The vertical part of the tangent space to $\mu^{-1}(0)/N$ is $T^*_{\rho^{\vee}(\omega)} \Bun_G(X)$ which surjects onto the cokernel $T^*_{\rho^{\vee}(\omega)} \Bun_N^{\omega}(X)$. So together the tangent spaces to $\Lambda_f''$ and $\mu^{-1}(0)/N$ span $T^*_{\rho^{\vee}(\omega)} \Bun_{G, N}^{\omega}(X, nx)$.
\end{proof}

Pulling back to $i':\Bun_N^{\omega}(X, nx) \rightarrow \Bun_G^{\omega}(X, nx)$ the intersection is no longer transverse but still clean.

\begin{prop}
Inside $T^* \Bun_G^{\omega}(X, nx)$, the shifted conormal bundle \beq \Lambda_f' := T^*_{\Bun_N^{\omega}(X, nx)} \Bun_G^{\omega}(X, nx) + df'\eeq intersects the global nilpotent cone $\Lambda'$ cleanly along smooth points. The intersection is $n \dim N$ dimensional.
\end{prop}
\begin{proof}
Both $\Lambda'$ and $\Lambda_f'$ live inside
\beq \mu^{-1}(\frn^{\perp}) := T^*\Bun^{\omega}_{G, N}(X, nx) \times_{\Bun^{\omega}_{G, N}(X, nx)} \Bun_G^{\omega}(X, nx) \subset T^*\Bun_G^{\omega}(X, nx)\eeq
and are pulled back from $\Lambda''$ and $\Lambda_f''$ respectively along \beq\label{Project} \xymatrix{\pi : \mu^{-1}(\frn^{\perp}) \ar[r] &  T^*\Bun^{\omega}_{G, N}(X, nx).}\eeq

By the previous Proposition \ref{KostNilp2}, $\Lambda'$ and $\Lambda'_f$ intersect transversely inside $\mu^{-1}(\frn^{\perp})$ and the dimension of intersection is $n \dim N$, the relative dimension of $\pi$.
Therefore they intersect cleanly inside the full $T^*\Bun_G^{\omega}(X, nx)$.
\end{proof}

Therefore by Proposition \ref{shift}, the Whittaker functional
\beq \phi_{f', \rho^{\vee}(\omega)} i'^! \simeq \phi_{F, \rho^{\vee}(\omega)} \eeq
is a shifted microstalk and
\beq
\phi_{f', \rho^{\vee}(\omega)} i'^![\dim_{\rho^{\vee}(\omega)} \Bun_{B^-}(X, nx) - n\dim N]
\eeq
is exact and commutes with Verdier duality. Descending from $\Bun_G^{\omega}(X, nx)$ back to  $\Bun_G(X)$ we get that $\phi_{f, \rho^{\vee}(\omega)}$ is exact after shifting by \beq 2n\dim N - n\dim G + \dim_{\rho^{\vee}(\omega)} \Bun_{B^-}(X, nx) - n\dim N = \dim_{\rho^{\vee}(\omega)} \Bun_{B^-}(X).\eeq Reassuringly, this expression is independent of $n$, the amount of uniformization. We have proved:

\begin{theorem}\label{whitmicrothm}
The shifted Whittaker functional
\beq \label{whitmicrothmmap}
\phi_{f, \rho^{\vee}(\omega)} i^! [\dim_{\rho^{\vee}(\omega)} \Bun_{B^-}(X)]
\eeq
calculates microstalk. In particular \eqref{whitmicrothmmap} is exact and commutes with Verdier duality.
\end{theorem}
The result that the Whittaker functional is exact and commutes with Verdier duality was also obtained in \cite{FR}.
Let's recall where all of the shifts came from in our arguments:
\begin{itemize}
\item The $2n \dim N$ shift appears from \eqref{VanishingShiftDifference} as the difference between vanishing cycles for $f$ on  $\Bun_N^{\omega}(X)$ versus $!$-pullback to $\Bun_N^{\omega}(X, nx)$ followed by vanishing cycles for the lifted function $f'$.
\item The $-n \dim G$ shift came from the fact that $!$-pullback along $\pi: \Bun_G^{\omega}(X, nx) \to \Bun_G(X)$ is not exact but $\pi^![-n\dim G]$ is.
\item The $\dim_{\rho^{\vee}(\omega)} \Bun_{B^-}(X, nx)$ shift is the dimension of $Y^{\leq 0}$ appearing in Proposition~\ref{shift}.
\item The $-n \dim N$ shift is minus the dimension of $\Lambda \cap (T^*_{Y^{> 0}}Y + df)$ appearing in Proposition~\ref{shift}.
\end{itemize}

\subsection{The Whittaker functional in the presence of tame ramification}\label{Tame}
In this section we extend Theorem \ref{whitmicrothm} the case of tame ramification at a finite subset of points $S \subset X$. Let $\Bun_N^{\omega(S)}(X, S)$ be the moduli of $B$-bundles such that the underlying $T$-bundle is $\rho^{\vee}(\omega(S))$ plus a trivialization of the fibers at the marked points $S$.  The Whittaker function is given by summing up \beq \xymatrix{f : \Bun_N^{\omega(S)}(X, S) \ar[r]& \Bun_{\BG_a}^{\omega(S)}(X, S) = \Bun_{\BG_a}^{\omega}(X) \ar[r]& \AA^1} \eeq over simple roots, see Section 2.5 of \cite{NY}.
There is a map \beq \label{RamifiedAttracting} \xymatrix{i: \Bun_N^{\omega(S)}(X, S) \ar[r]& \Bun_{G, N^-}(X, S)}\eeq to the moduli of $G$-bundle with $N^-$-reductions at $S$.

The cotangent space \beq T^*\Bun_{G, N^-}(X, S) = \{E, F_S, \sigma \in H^0(\frg^*_E \otimes \omega(S)) | \res_S(\sigma) \in \frb^*_{F_S} \} \eeq is the moduli of $G$-bundles $E$ with an $N^-$-reduction $F_S$ at $S$ plus a Higgs field $\sigma \in H^0(\frg^*_E \otimes \omega(S))$ whose residue at $S$ is in $\frb^*$ with respect to $F_S$. The Hitchin map \beq \xymatrix{h: T^*\Bun_{G, N^-}(X, S) \ar[r] & \frc^*_{G, N}(X, S)}\eeq sends a Higgs field to its characteristic polynomial plus an ordering of the eigenvalues at the points of $S$. Let $\Lambda = h^{-1}(0)$ be the nilpotent cone in $T^*\Bun_{G, N^-}(X, S)$.

The cotangent space $T^* \Bun_{G, N^-}(X, S)$ is represented by maps $X \rightarrow \frg^*/G \times \BG_m$ such that the underlying $\BG_m$-bundle is $\omega(S)$, plus a lifting at the marked points $S \rightarrow \frb^*/N^-$. We identified $\frb^* = (\frn^-)^{\perp}$ using the Killing form. Below we list the other relevant cotangent spaces together with the pairs of spaces representing them:
\begin{center}
\begin{tabular}{ l l l }
 $T^* \Bun_{G, N^-}(X, S)$ & $\frg^*/G \times \BG_m$ & $\frb^*/N^-$ \\
 $T^* \Bun_{G, N^-}(X, S)|_{\Bun_N^{\omega(S)}(X, S)}$ & $\frg^*/N \rtimes \BG_m$ & $\frb^*$ \\
 $T^*\Bun_N^{\omega(S)}(X)$ & $\frn^*/N \rtimes \BG_m$ & $\pt/N$ \\
 $T^*\Bun_N^{\omega(S)}(X, S)$ & $\frn^*/N \rtimes \BG_m$ & $\frn^*$ \\
  $df + T^*_{\Bun_N^{\omega(S)}(X, S)} \Bun_{G, N^-}(X, S)$ & $(\psi + \frn^{\perp})/N \rtimes \BG_m$ & $\frt^*$ \\
$\frc^*_{G, N}(X, S)$ & $\frc^*/\BG_m$ & $\frt^*$ \\
\end{tabular}
\end{center}

\begin{theorem}
The shifted Whittaker functional \beq \label{RamifiedFunctional} \xymatrix{\phi_{f, \rho^{\vee}(\omega)}i^![\dim_{\rho^{\vee}(\omega)} \Bun_{B^-, N^-}(X, S)]: \Sh_{\Lambda}(\Bun_{G, N^-}(X, S)) \ar[r]& \Vect }\eeq calculates microstalk. In particular \eqref{RamifiedFunctional} is exact and commutes with Verdier duality.
\end{theorem}
\begin{proof}
By looking at cotangent spaces we see that $\Bun_N^{\omega(S)}(X, S)$ is the full attracting locus in $\Bun_{G, N^-}(X, S)$. Moreover the shifted conormal maps isomorphically to the Hitchin base under  \beq \xymatrix{df + T^*_{\Bun_N^{\omega(S)}(X, S)} \Bun_{G, N^-}(X, S) \ar[r] & T^* \Bun_{G, N^-}(X, S) \ar[r] & \frc^*_{G, N}(X, S)}\eeq because the above composition is represented by the map of pairs \beq \xymatrix{((\psi + \frn^{\perp})/(N \rtimes \BG_m), \frt^*) \ar[r]& (\frg^*/(G \times \BG_m), \frb^*/N^-) \ar[r]& (\frc^*/\BG_m, \frt^*).}\eeq Therefore the shifted conormal intersects the global nilpotent cone $\Lambda$ transversely at a single smooth point. The result now follows by uniformizing $\Bun_{G, N^-}(X, S)$ by a scheme and then applying Proposition \ref{shift}.
\end{proof}

\bibliographystyle{alpha}
\bibliography{refs}

\end{document}